\documentclass{amsart}
\usepackage{geometry}                % See geometry.pdf to learn the layout options. There are lots.
\usepackage{graphicx}
\usepackage{amssymb}
\usepackage{epstopdf,amsmath}
\newtheorem{prop}{Proposition}[section]
\newtheorem{lemma}[prop]{Lemma}

\newtheorem{theorem}[prop]{Theorem}

\DeclareMathOperator{\Dim}{Dim}
\DeclareMathOperator{\Poly}{Poly}

\DeclareMathOperator{\Sign}{Sign}

\newcommand{\ZZ}{\mathbb{Z}}
\newcommand{\RR}{\mathbb{R}}

\title{Polynomial partitioning for a set of varieties}

\author{Larry Guth}

\begin{document}

\begin{abstract} Given a set $\Gamma$ of low-degree k-dimensional varieties in $\RR^n$, we prove that for any $D \ge 1$, there is a non-zero polynomial $P$ of degree at most $D$ so that each component of $\RR^n \setminus Z(P)$ intersects $O(D^{k-n} |\Gamma|)$ varieties of $\Gamma$.

\end{abstract}

\maketitle

Recently polynomial partitioning has become a valuable technique in incidence geometry.  In particular the following partitioning theorem has had some important applications.

\begin{theorem} \label{polypartgk} (Theorem 4.1 in \cite{GK}) If $X$ is a finite set of points in $\RR^n$ and $D \ge 1$, then there is a non-zero polynomial $P$ of degree at most $D$ so that each component of $\RR^n \setminus Z(P)$ contains at most $C_n D^{-n} |X|$ points of $X$.
\end{theorem}

This theorem is a kind of equidistribution result.  Milnor \cite{M} and Thom \cite{T} proved that $\RR^n \setminus Z(P)$ has at most $C_n D^n$ connected components (see also Theorem \ref{numbercells} below).  If all the points of $X$ were in $\RR^n \setminus Z(P)$, then the conclusion of Theorem \ref{polypartgk} would imply that the points were roughly equidistributed among the components of $\RR^n \setminus Z(P)$.  It is important to note, however, that some or all of the points of $X$ are allowed to lie in $Z(P)$.  For example, if $X$ is a large set of points in a hyperplane in $\RR^n$, then $Z(P)$ could be that hyperplane.  

Katz and the author used Theorem \ref{polypartgk} in \cite{GK} to study the incidence geometry of lines in $\RR^3$, leading to new bounds for the distinct distance problem in the plane.  In \cite{KMS}, Kaplan, Matou\u{s}ek, and Sharir used it to reprove several classical theorems in incidence geometry, including the Szemer\'edi-Trotter theorem.  In \cite{ST}, Solymosi and Tao used it to study the incidence geometry of $k$-planes in $\RR^n$.  Theorem \ref{polypartgk} has been applied to other problems in incidence geometry by Sharir, Sheffer, and Zahl \cite{SSZ}, by Sharir and Solomon \cite{SS}, by Kaplan, Matou\u{s}ek, Safernov\'a, and Sharir \cite{KMSS}, and by Zahl \cite{Z}.  

In this paper, we consider a generalization of Theorem \ref{polypartgk}.  Instead of a finite set of points $X$, we consider a finite set of algebraic varieties.  For example, we may consider a set of lines, a set of $k$-planes, a set of circles, etc.  I don't have any immediate applications of this generalized partitioning theorem, but because of the many recent applications of Theorem \ref{polypartgk}, I hope that this generalization will also be useful in incidence geometry.

Suppose that $\Gamma$ is a set of $k$-dimensional varieties in $\RR^n$.  We would like to partition $\RR^n$ with a degree $D$ polynomial $P$ so that each component of $\RR^n \setminus Z(P)$ intersects only a small number of the varieties of $\Gamma$.   As a starting point, we consider a single variety $\gamma$, and we ask how many components of $\RR^n \setminus Z(P)$ the variety $\gamma$ can intersect.  This question was studied by Barone and Basu \cite{BB}.  Solymosi and Tao gave a nice exposition of a less precise result in the appendix of their paper \cite{ST}.

\begin{theorem} \label{numbercells} (Theorem A2 in \cite{ST}) Suppose $\gamma$ is a $k$-dimensional variety in $\RR^n$ defined by $m$ polynomial equations $p_j(x) = 0$ each of degree at most $d$.   If $P$ is a polynomial of degree at most $D$, then $\gamma$ intersects at most $C(d, m, n) D^k$ different components of $\RR^n \setminus Z(P)$.
\end{theorem}

(Remark. For the definition of a $k$-dimensional variety, see Section 4 of \cite{ST}.)

Suppose that $P$ was a degree $D$ polynomial and that $\RR^n \setminus Z(P)$ consisted of $\sim D^n$ cells and that each cell intersected the same number of varieties $\gamma \in \Gamma$.  Then Theorem \ref{numbercells} would imply that each connected component of $\RR^n \setminus Z(P)$ intersected at most $C(d,m,n) D^{k-n} |\Gamma|$ varieties $\gamma \in \Gamma$.  We prove that there is a polynomial $P$ of degree at most $D$ that obeys this bound.

\begin{theorem} \label{polypartk} Suppose $\Gamma$ is a set of $k$-dimensional varieties in $\RR^n$, each defined by at most $m$ polynomial equations of degree at most $d$.  For any $D \ge 1$, there is a non-zero polynomial $P$ of degree at most $D$, so that each connected component of $\RR^n \setminus Z(P)$ intersects at most $C(d,m,n) D^{k-n} |\Gamma|$ varieties $\gamma \in \Gamma$.  
\end{theorem}

Let us sketch the proof of Theorem \ref{polypartgk} and explain the new difficulty that comes up in proving Theorem \ref{polypartk}.  To prove Theorem \ref{polypartgk}, we find a sequence of polynomials $P_1, P_2,$ etc.  The final polynomial $P$ will be the product $\prod_{j} P_j$.  We choose $P_1$ to bisect $X$: in other words, we choose $P_1$ so that (at most) half of the points of $X$ lie in $\{ P_1 > 0 \}$ and (at most) half of the points of $X$ lie in $\{ P_1 < 0 \}$.  Then we choose $P_2$ to bisect each of these sets.  In other words, $P_2$ bisects the set $\{ x \in X | P_1(x) > 0 \}$ and the set $\{ x \in X | P_1(x) < 0 \}$.  The sign conditions of $P_1$ and $P_2$ determine four regions, and each region contains at most a quarter of the points of $X$.  At each step, we can find the polynomial $P_j$ by the Stone-Tukey ham sandwich theorem \cite{ST} which gives a good estimate for the degree of $P_j$.  The Stone-Tukey ham sandwich theorem in turn follows from the Borsuk-Ulam theorem.  

Suppose that we take this approach to try to prove Theorem \ref{polypartk}.  Josh Zahl pointed out to me the following issue.  Recall that $\Gamma$ is a set of $k$-dimensional varieties.  For concreteness, suppose that $\Gamma$ is a set of $100$ lines.  We first try to choose a polynomial $P_1$ so that the number of lines of $\Gamma$ that enter $\{ P_1 > 0 \}$ is equal to the number of varieties of $\Gamma$ that enter $\{ P_1 < 0 \}$.  We could do something close to this with the ham sandwich theorem.  Notice that a line may enter both regions, or it may lie in one region.  So it could happen that 50 of the 100 lines enter each region, or it could happen that all 100 lines enter each region, or anything in between.  This issue isn't a problem on the first step, but it will become a problem when we try to choose $P_2$.  

Let $\Gamma_+ \subset \Gamma$ be the set of lines of $\Gamma$ that enter $\{ P_1 > 0 \}$ and let $\Gamma_- \subset \Gamma$ be the set of lines that enter $\{ P_1 < 0 \}$.  Suppose for this example that $|\Gamma_+| = |\Gamma_-| = 80$.  
Next we try to choose $P_2$ to ``bisect'' $\Gamma_+$ and $\Gamma_-$.
In other words, we want the number of lines of $\Gamma_+$ that enter $\{ P_2 > 0 \}$ to be equal to the number of lines of $\Gamma_+$ that enter $\{ P_2 < 0 \}$, and similarly for $\Gamma_-$.  We define $\Gamma_{++} \subset \Gamma_+$ to be the set of lines of $\Gamma$ that enter the region $\{ P_1 > 0, P_2 > 0\}$, we define $\Gamma_{+-} \subset \Gamma_+$ to be the set of lines of $\Gamma$ that enters the region $ \{ P_1 > 0, P_2 < 0 \}$, and similarly we define $\Gamma_{-+}, \Gamma_{--} \subset \Gamma_-$.  If $P_2$ bisects $\Gamma_+$ and $\Gamma_-$, then we get $| \Gamma_{++}| = | \Gamma_{+-}|$ and $| \Gamma_{-+}| = | \Gamma_{--}|$.  But now the following issue arises.  If $|\Gamma_+| = |\Gamma_-| = 80$, it may happen that $| \Gamma_{++}| = | \Gamma_{+-}| = 40$ and $| \Gamma_{-+}| = | \Gamma_{--}| = 80$.  The four sets of lines are not equal!

To explain our approach to Theorem \ref{polypartk}, suppose we just wanted to choose two polynomials $P_1, P_2$ so that $|\Gamma_{++}| = |\Gamma_{+-}| = |\Gamma_{-+}| = |\Gamma_{--}|$.  Instead of choosing $P_1$ and then $P_2$, we simultaneously look for polynomials $P_1, P_2$ so that the following three equations hold:

$$ | \Gamma_{++}| + |\Gamma_{+-}| = |\Gamma_{--}| + |\Gamma_{-+}|, $$

$$ |\Gamma_{++}| + |\Gamma_{-+}| = |\Gamma_{--}| + |\Gamma_{+-}|, $$

$$ |\Gamma_{++}| + |\Gamma_{--}| = |\Gamma_{-+}| + |\Gamma_{+-}|. $$

\noindent  We can find polynomials $P_1, P_2$ that obey (a continuous approximation of) these equations by using a cousin of the Borsuk-Ulam theorem.  These three equations then imply that $|\Gamma_{++}| = |\Gamma_{+-}| = |\Gamma_{-+}| = |\Gamma_{--}|$.   The cousin of the Borsuk-Ulam theorem that we use is new in the sense that it has not been written down before, but it follows from standard arguments in topology.  

In Section 1, we state the cousin of the Borsuk-Ulam theorem that we need.  In Section 2 we give an outline of the proof of our main theorem.  In Section 3, we give the full proof.  In Section 4, we prove the cousin of the Borsuk-Ulam theorem stated in Section 1.  

\vskip10pt

{\bf Acknowledgements.} I would like to thank Josh Zahl for interesting conversations related to this paper.  I would also like to thank the referee for some helpful suggestions.

\section{A variation of the Borsuk-Ulam theorem}

Theorems \ref{polypartgk} and \ref{polypartk} use topological arguments.
Theorem \ref{polypartgk} has a short proof using the Stone-Tukey ham sandwich theorem (\cite{StTu}), which in turn follows from the Borsuk-Ulam theorem.  Our proof of Theorem \ref{polypartk} uses a cousin of the Borsuk-Ulam theorem. 

For context, we recall the Borsuk-Ulam theorem (cf. Chapter 2.6 of \cite{GP}).

\begin{theorem} Suppose that $f: S^n \rightarrow \RR^n$ is a continuous map obeying the antipodal condition $f( - x) = - f(x)$.  Then there is a point $x \in S^n$ where $f(x) = 0$.
\end{theorem}

To state our cousin of the Borsuk-Ulam theorem, we need some definitions.  

Let $X_s = \prod_{j=1}^s S^{2^{j-1}}$.  We note that $\Dim X_s = 2^s -1$.  We write a point $x \in X_s$ as $(x_1, ..., x_s)$ with $x_j \in S^{2^{j-1}}$.  We define the coordinate-flipping operation $Fl_j: X_s \rightarrow X_s$ by changing the sign of the $j^{th}$ coordinate:

$$ Fl_j( x_1, ..., x_{j-1}, x_j, x_{j+1}, ..., x_s) = ( x_1, ..., x_{j-1}, - x_j, x_{j+1}, ..., x_s). $$

For each $v \in \ZZ_2^s \setminus \{ 0 \}$, suppose that $f_v: X_s \rightarrow \RR$ is a continuous function that obeys the following antipodal-type condition:

\begin{equation}\label{f_vantipode}
 f_v (Fl_j x) = (-1)^{v_j} f_v (x) \textrm{ for all } j = 1, ..., s. 
\end{equation}

Note that we have $2^s - 1$ functions $f_v: X_s \rightarrow \RR$.  The dimension of $X_s$ is also $2^s - 1$. 

\begin{theorem} \label{buvar} If $f_v: X_s \rightarrow \RR$ are continuous functions that obey equation \ref{f_vantipode} for each $v \in \ZZ_2^s \setminus \{ 0 \}$, then there exists a point $x \in X_s$ where $f_v(x) = 0$ for all $v \in \ZZ_2^s \setminus \{ 0 \}$.  
\end{theorem}

We prove Theorem \ref{buvar} in Section \ref{secproofbuvar}.  First, we explain how to use Theorem \ref{buvar} to prove our main result, Theorem \ref{polypartk}.

\section{Outline of the proof of the partitioning theorem}

Let $\Poly_D (\RR^n)$ be the vector space of polynomials on $\RR^n$ with degree at most $D$.  For fixed $n$, $\Dim \Poly_D (\RR^n) \sim_n D^n$.  For each $j$, we choose $D_j$ so that $\Dim \Poly_{D_j}(\RR^n) > 2^{j-1}$.  We have $D_j \lesssim 2^{j/n}$.  

We pick a subspace of $\Poly_{D_j}(\RR^n)$ with dimension $2^{j-1} + 1$, and we identify $S^{2^{j-1}}$ with the unit sphere in this subspace.  In this way we get an embedding

$$ X_s \subset \prod_{j=1}^s \Poly_{D_j} (\RR^n) . $$

We let $D = \sum_j D_j \lesssim 2^{s/n}$.

If $\vec P = (P_1, ..., P_s) \in X_s \subset \prod_{j=1}^s \Poly_{D_j} (\RR^n)$, then for any $w \in \ZZ_2^s$, we define the cell

$$O(\vec P, w) := \{ x \in \RR^n | \Sign P_j(x) = (-1)^{w_j} \textrm{ for all } j \in [1, ..., s] \}. $$

Note that $P_{prod} = \prod_{j=1}^s P_j$ has degree at most $D$.  We see that $\RR^n \setminus Z(P_{prod})$ is the disjoint union of the cells $O(\vec P, w)$.  The number of $w \in \ZZ_2^s$ is $2^s \sim D^n$.  For a good choice of $\vec P \in X_s$, we will show that each of these cells does not intersect too many varieties of $\Gamma$.  

Recall that $\Gamma$ is a finite set of $k$-dimensional varieties.  For $\gamma \in \Gamma$, we let $I^{\gamma}(\vec P, w)$ be the indicator function:

$$ I^{\gamma}(\vec P, w) = 1 \textrm{ if } O(\vec P, w) \cap \gamma \textrm{ is non-empty} ; $$

$$ I^{\gamma}(\vec P, w) = 0 \textrm{ if } O(\vec P, w) \cap \gamma \textrm{ is empty}. $$

Note that $\sum_{\gamma \in \Gamma} I^{\gamma}(\vec P, w)$ is the number of varieties $\gamma \in \Gamma$ that intersect $O(\vec P, w)$.  

Define $G_v(\vec P)$ as follows:

$$ G_v(\vec P) := \sum_{w \in \ZZ_2^s, w \cdot v = 0}  \sum_{\gamma \in \Gamma} I^{\gamma}(\vec P, w) - \sum_{w \in \ZZ_2^s, w \cdot v = 1}  \sum_{\gamma \in \Gamma} I^{\gamma}(\vec P, w). $$

The function $G_v$ obeys the antipodal-type condition in equation \ref{f_vantipode}: $ G_v( Fl_j \vec P) = (-1)^{v_j} G_v (\vec P)$.  
The functions $G_v: X_s \rightarrow \RR$ are not continuous: the problem is that the indicator function $I^{\gamma}(\vec P, w)$ is not continuous in $\vec P$.  Therefore, Theorem \ref{buvar} does not apply to $G_v$.  Nevertheless, to get a feel for the proof, let us pretend for a moment that $G_v$ was continuous.  Then Theorem \ref{buvar} would imply that there exists $\vec P \in X_s$ so that $G_v(\vec P) = 0$ for all $v \in \ZZ_2^s \setminus \{ 0 \}$.  Then a short calculation would show that $\sum_{\gamma \in \Gamma} I^{\gamma}(\vec P, w) $ is independent of $w$.  (This calculation is explained in Lemma \ref{equidistw} below.)  In other words, each of the $2^s$ cells $O(\vec P, w)$ would intersect the same number of varieties $\gamma \in \Gamma$.  Since each variety $\gamma$ can enter at most $C(d,m,n) D^k$ cells, the number of varieties intersecting each cell would be at most $2^{-s} C(d,m,n) D^k |\Gamma| \le C_n C(d,m,n) D^{k-n} |\Gamma|$.  This would  prove Theorem \ref{polypartk}.

The fact remains that $G_v: X_s \rightarrow \RR$ is not continuous.  We will consider continuous approximations of $G_v$.  

\section{Continuous approximation}

We will define $I_\delta^{\gamma}(\vec P, w)$ for each $\delta > 0$.  The functions $I_\delta$ are a continuous approximation of the indicator function $I^{\gamma}(\vec P, w)$.  More precisely, we will define $I_\delta$ with the following properties.

\begin{lemma} \label{I_deltaprop} For each $\delta > 0$, $\gamma \in \Gamma$, $w \in \ZZ_2^s$, and $\vec P \in X_s$, we will define $I_\delta^{\gamma}(\vec P, w) \in \RR$ with the following properties.

\begin{enumerate} 

\item For each $\delta > 0$, $w \in \ZZ_2^s$ and $\gamma \in \Gamma$, $I_\delta^{\gamma}(\vec P, w)$ is continuous in $\vec P\in X_s$.

\item $0 \le I_\delta^{\gamma}(\vec P, w) \le 1$.

\item If $\gamma \cap O(\vec P, w)$ is empty, then $I_\delta^{\gamma}(\vec P, w) = 0$.

\item If $\delta_i \rightarrow 0$ and $\vec P_i \rightarrow \vec P$ in $X_s$ and $\gamma \cap O(\vec P, w)$ is non-empty, then

$$ \lim_{i \rightarrow \infty} I_{\delta_i}^{\gamma}(\vec P_i, w) = 1. $$

In other words, $I^{\gamma}(\vec P, w) \le \liminf_{i \rightarrow \infty} I_{\delta_i}^{\gamma}(\vec P_i, w)$.  

\end{enumerate}

\end{lemma}

Using this Lemma, we finish the proof of Theorem \ref{polypartk}.  Then we define $I_\delta^{\gamma}(\vec P, w)$ and prove Lemma \ref{I_deltaprop}.  

First we define a continuous version of $G_v$, replacing the indicator $I$ with the continuous version $I_\delta$:

$$ f_{\delta, v}(\vec P) := \sum_{w \in \ZZ_2^s, w \cdot v = 0}  \sum_{\gamma \in \Gamma} I_\delta^{\gamma}(\vec P, w) - \sum_{w \in \ZZ_2^s, w \cdot v = 1}  \sum_{\gamma \in \Gamma} I_\delta^{\gamma}(\vec P, w). $$

Since $I_\delta$ is continuous in $\vec P \in X_s$, it follows that each function $f_{\delta,v}: X_s \rightarrow \RR$ is continuous.  Moreover, each function obeys the antipodal-type condition \ref{f_vantipode}: $f_{\delta, v} (Fl_j \vec P) = (-1)^{v_j} f_{\delta,v}(\vec P)$.  Therefore, for each $\delta > 0$, Theorem \ref{buvar} implies that there is a $\vec P_\delta \in X_s$ so that $f_{\delta,v} (\vec P_\delta) = 0$ for all $v \in \ZZ_2^s \setminus \{ 0 \}$.

\begin{lemma} \label{equidistw} For $\vec P_\delta$ as above, $\sum_{\gamma \in \Gamma} I_\delta^{\gamma}(\vec P_\delta, w)$ is constant in $w \in \ZZ_2^s$. 
\end{lemma}

\begin{proof} Let $u$ be a non-zero vector in $\ZZ_2^s$.  We will show that

$$\sum_{\gamma \in \Gamma} I_\delta^{\gamma}(\vec P_\delta, u) = \sum_{\gamma \in \Gamma} I_\delta^{\gamma}(\vec P_\delta, 0). $$

We know that $f_{\delta, v} (\vec P_\delta) = 0$ for all non-zero $v \in \ZZ_2^s$.  For each $v \in \ZZ_2^s \setminus \{ 0 \}$, we have:

$$ \sum_{w \in \ZZ_2^s, w \cdot v = 0}  \sum_{\gamma \in \Gamma} I_\delta^{\gamma}(\vec P_\delta, w) = \sum_{w \in \ZZ_2^s, w \cdot v = 1}  \sum_{\gamma \in \Gamma} I_\delta^{\gamma}(\vec P_\delta, w). $$

We sum this equation over the set of $v \in \ZZ_2^s$ obeying $v \cdot u = 1$.  There are $2^{s-1}$ such $v$, and they are each non-zero.  So we get:

$$ \sum_{v \cdot u = 1} \left( \sum_{w \in \ZZ_2^s, w \cdot v = 0}  \sum_{\gamma \in \Gamma} I_\delta^{\gamma}(\vec P_\delta, w) \right) =  \sum_{v \cdot u = 1} \left( \sum_{w \in \ZZ_2^s, w \cdot v = 1}  \sum_{\gamma \in \Gamma} I_\delta^{\gamma}(\vec P_\delta, w) \right). $$

To simplify the sum, we define $N_0(w)$ to be the number of $v$ with $v \cdot u = 1$ and $v \cdot w = 0$.  We define $N_1(w)$ to be the number of $v$ with $v \cdot u = 1$ and $v \cdot w = 1$.  With this language, the sum becomes:

$$ \sum_{w \in \ZZ_2^s} N_0(w) \sum_{\gamma \in \Gamma} I_\delta^{\gamma}(\vec P_\delta, w) = \sum_{w \in \ZZ_2^s} N_1(w) \sum_{\gamma \in \Gamma} I_\delta^{\gamma}(\vec P_\delta, w) . $$

Next we evaluate $N_0(w)$ and $N_1(w)$, which makes the formula much simpler.  Recall that $N_0(w)$ is the number of solutions $v$ to the equations $v \cdot u =1$ and $v \cdot w = 0$.  Similarly, $N_1(w)$ is the number of solutions $v$ to the equations $v \cdot u = 1$ and $v \cdot w = 1$.  First, $N_0(0) = 2^{s-1}$ and $N_1(0) = 0$.  Second, $N_0(u) = 0$, and $N_1(u) = 2^{s-1}$.  Finally, if $w$ is not equal to $0$ or $u$, then $w, u$ are linearly independent, and so $N_0(w) = N_1(w) = 2^{s-2}$.  Subtracting the terms in common on both sides and dividing by $2^{s-1}$, the last equation reduces to:

$$ \sum_{\gamma \in \Gamma} I_\delta^{\gamma}(\vec P_\delta, 0) = \sum_{\gamma \in \Gamma} I_\delta^{\gamma}(\vec P_\delta, u). $$

\end{proof}

By Property 3 of Lemma \ref{I_deltaprop}, we know that if $\gamma \cap O(P, w)$ is empty, then $I_\delta^{\gamma}(\vec P, w) = 0$.  Also, by Proposition \ref{numbercells} each variety $\gamma$ enters at most $C(d,m,n) D^k$ of the cells $O(P_\delta, w)$.  Therefore, for any $\vec P \in X_s$,

$$ \sum_{w \in \ZZ_2^s} \sum_{\gamma \in \Gamma} I_\delta^{\gamma}(\vec P, w) \le C(d,m,n) D^k |\Gamma|.$$

By Lemma \ref{equidistw}, $ \sum_{\gamma \in \Gamma} I_\delta^{\gamma}(\vec P_\delta, w)$ is independent of $w$, and so for each $w \in \ZZ_2^s$, 

$$ \sum_{\gamma \in \Gamma} I_\delta^{\gamma}(\vec P, w) \le 2^{-s} C(d,m,n) D^k |\Gamma| \le C_n C(d,m,n) D^{k-n} |\Gamma|. $$

Since $X_s$ is compact, there is a subsequence of $\vec P_\delta$ that converges to a limit $\vec P$ as $\delta \rightarrow 0$.  By Property 4 of Lemma \ref{I_deltaprop}, we know that for each $\gamma \in \Gamma$ and $w \in \ZZ_2^s$,

$$ I^{\gamma} (\vec P, w) \le \liminf_{\delta \rightarrow 0} I_\delta^{\gamma} (\vec P_\delta, w).   $$

Summing over $\gamma \in \Gamma$, we see that for each $w \in \ZZ_2^s$,

$$ \sum_{\gamma \in \Gamma}  I^{\gamma} (\vec P, w) \le \liminf_{\delta \rightarrow 0} \sum_{\gamma \in \Gamma}  I_\delta^{\gamma} (\vec P_\delta, w) \le C_n C(d,m,n) D^{k-n} |\Gamma|. $$

In other words, each cell $O(\vec P, w)$ intersects at most $C_n C(d,m,n) D^{k-n} |\Gamma|$ varieties $\gamma \in \Gamma$.  
This is the conclusion of Theorem \ref{polypartk}.  It only remains to construct the continuous approximation $I_\delta^{\gamma} (\vec P, w)$ and check the four properties in Lemma \ref{I_deltaprop}.

\subsection{Constructing $I_\delta$} For each $\epsilon > 0$ we define a continuous function $\eta_\epsilon: \RR \rightarrow \RR$ so that

\begin{itemize}

\item If $t \le \epsilon$, $\eta_\epsilon(t) = 0$.

\item If $t \ge 2 \epsilon$, $\eta_\epsilon(t) = 1$.

\item For all $t \in \RR$, $0 \le \eta_\epsilon(t) \le 1$.

\end{itemize}

Next we define functions $\epsilon(\delta)$ and $R(\delta)$ so that as $\delta \rightarrow 0$, $\epsilon(\delta) \rightarrow 0$ slowly, and $R(\delta) \rightarrow \infty$ slowly.  We will make this more precise below.  

We write $N_\delta \gamma$ for the $\delta$-neighborhood of $\gamma$, and $(P_1, ..., P_s)$ for the components of $\vec P$.

Now we can define $I_\delta^\gamma(\vec P, w)$: 

$$ I_\delta^\gamma(\vec P, w) = \eta_\epsilon \left(     \int_{N_\delta \gamma \cap O(\vec P, w) \cap B_R} \eta_\epsilon (\min |P_i| ) \delta^{-n}            \right). $$

Since $\eta_\epsilon$ is a continuous function, the integrand is continuous in $\vec P \in X_s$.  The domain of integration is also continuous, in the sense that if $\vec P_i \rightarrow \vec P$, then the volume of the symmetric difference of $O(\vec P_i, w) \cap B_R$ and $O(\vec P, w) \cap B_R$ goes to zero.  Therefore, the integral is a continuous function of $\vec P$, and so $I_\delta^\gamma(\vec P, w)$ is a continuous function of $\vec P \in X_s$.

Since $0 \le \eta_\epsilon(y) \le 1$, it follows immediately that $0 \le I_\delta^\gamma(\vec P, w) \le 1$.

Now we consider Property 3.  Suppose that $\gamma \cap O(\vec P, w)$ is empty.  If we choose $\epsilon(\delta)$ and $R(\delta)$ carefully, then we will show that on the domain of integration $N_\delta \gamma \cap O(\vec P, w) \cap B_{R(\delta)}$, $\min |P_i| \le \epsilon$, and so $\eta_\epsilon ( \min |P_i|) = 0$.  This will show that the integral is zero and so $I_\delta^\gamma(\vec P, w) = 0$.

Let $x \in N_\delta \gamma \cap O(\vec P, w) \cap B_{R(\delta)}$.  There must be another point $\bar x \in \gamma$ with $|x - \bar x| \le \delta$.  Since $\gamma \cap O(\vec P, w)$ is empty, there must be some $i$ so that $\Sign P_i (\bar x) \not= \Sign P_i(x)$.  (To be precise, we mean that either $P_i (\bar x) = 0$ and $P_i(x) \not= 0$, or else $P_i(\bar x)$ and $P_i(x)$ are both non-zero and have opposite signs.)  Therefore, there must be a point $y$ on the closed segment from $x$ to $\bar x$ where $P_i(y) = 0$.

Now we choose $\epsilon(\delta) \rightarrow 0$ and $R(\delta) \rightarrow \infty$ slowly enough that

$$ \max_{Q \in X_s} \max_{x \in B_{R(\delta) +1}} |\nabla Q(x)| \delta < \epsilon(\delta). $$

In particular, along the segment from $y$ to $x$, we see that $|\nabla P_i| \delta < \epsilon$.  Since the segment has length at most $\delta$, and since $P_i(y) = 0$, we see that $|P_i(x)| \le \epsilon$ as desired.  This proves Property 3.

Now we consider Property 4.  Suppose that $O(\vec P, w)$ contains a point $q \in \gamma$.  Consider a sequence of numbers $\delta \rightarrow 0$.  Suppose that as $\delta \rightarrow 0$, $\vec P_\delta \rightarrow \vec P$ in $X_s$.  For all $\delta$ sufficiently small, the following things happen.  The ball $B_{\delta}(q) \subset O(\vec P_\delta, w)$.  On $B_{\delta}(q)$, $\min |P_i| \ge c > 0$ for some constant $c > 0$.  On $B_{\delta}(q)$, $\min |P_{\delta,i}| \ge c/2 > 0$.  So on $B_{\delta}(q)$, $\eta_\epsilon( \min |P_{\delta,i}| ) = 1$.  The ball $B_{R(\delta)}$ contains $B_{\delta}(q)$.  And so 

$$ \int_{N_\delta \gamma \cap O(\vec P_\delta, w) \cap B_R} \eta_\epsilon (\min |P_{\delta, i}| ) \delta^{-n}  \ge c' > 0. $$

Therefore, $I_\delta^\gamma(\vec P_\delta, w) = 1$ for all $\delta$ sufficiently small.   This proves Property 4 and finishes the proof of Lemma \ref{I_deltaprop}.

\section{Proof of Theorem \ref{buvar}}  \label{secproofbuvar}

In this section, we prove Theorem \ref{buvar}, the topological input to our argument.  
Theorem \ref{buvar} is a cousin of the Borsuk-Ulam theorem, and we will adapt one of the standard proofs of the Borsuk-Ulam theorem.

The following topology theorem, due to Brouwer, is the main tool in the proof.

\begin{theorem} \label{brouwer} (Brouwer 1909) Suppose that $X^N$ is a compact $N$-dimensional possibly with boundary.  Suppose that $f, g: X \rightarrow \RR^N$ are smooth functions which agree on the boundary $\partial X$.  Suppose that $f$ and $g$ do not vanish on $\partial X$, and suppose that 0 is a regular value for both $f$ and $g$.  Then

$$ \# f^{-1}(0) = \# g^{-1}(0) \textrm{ modulo 2}. $$

\end{theorem}

Here we write $\# f^{-1} (0)$ for the number of points in the set $f^{-1}(0)$.  Part of the conclusion of the theorem is that this number is always finite.  

The proof of Theorem \ref{brouwer} can be found in Milnor's introduction to differential topology \cite{M2}.  The result appears as the Homotopy Lemma on page 21 of \cite{M2}.  The result is stated there for the case that $X$ has no boundary, but the proof applies word for word to our setting: $X$ has a boundary, $f$ and $g$ agree on the boundary, and $f$ and $g$ don't vanish on the boundary.  
The book \cite{M2} is a very readable and engaging introduction to the subject.  

Let us recall the definition of a regular value.  For any $x \in X$, the derivative $df_x$ is a linear map from the tangent space $T_x X$ to $\RR^N$.  A point $y \in \RR^N$ is a regular point if, for every $x \in f^{-1}(y)$, $df_x$ is surjective.  Here is a simple example to illustrate the definition.  If $f: [-1,1] \rightarrow \RR$ is the map $f(x) = x^2$, then $df_x: \RR \rightarrow \RR$ is the linear map $ df_x(v) = 2x v$.  The map $df_x$ is surjective if and only if $x \not= 0$.  Now $f(0) = 0$, and so 0 is not a regular value of $f$, but every other $y \in \RR$ is a regular value of $f$.  We remark that if $f^{-1}(y)$ is empty, then $y$ is a regular value of $f$.  

Now let us try to give a little intuition for this theorem by considering low-dimensional examples.  Suppose that $f,g: [-1, 1] \rightarrow \RR$ with boundary values $f( \pm 1) = g(\pm 1) = 1$.  Suppose that $f$ is simply the function 1.  The function $f$ does not vanish at all and so $\# f^{-1}(0) = 0$.  Now let $g$ be the function $a x^2 + (1-a)$ for some constant $c$.  If $a < 1$, then $\# g^{-1}(0) = 0 = \# f^{-1}(0)$.  If $a > 1$, then $\# g^{-1}(0) = 2$.  In this case, $\# f^{-1}(0)$ is not equal to $\# g^{-1}(0)$, but they are equal modulo 2.  Now consider the case $a=1$.  In this case, $\# g^{-1}(0)$ is 1, which does not agree with $\# f^{-1}(0)$ modulo 2.  But if $a=1$, then 
$g(x) = x^2$ is the function we considered in the last paragraph.  In this case, 0 is not a regular value of $g$, and so the Theorem does not apply.  The problem is that the graph of the function $g$ is tangent to the $x$-axis at $x=0$ instead of crossing the $x$-axis.  When we say that 0 is a regular value of $g$, we rule out this problem with tangency.  Hopefully this discussion gives some intuition for the role of regular values in the Theorem.

Another basic point about regular values is that non-regular values are rare.  Sard's theorem states that for a smooth map $f: X \rightarrow \RR^N$, almost every $y \in \RR^N$ is a regular value.  (See Chapter 2 of \cite{M2}.)  Similarly, any smooth map $f: X \rightarrow \RR^N$ can be slightly perturbed to a map $\tilde f$ so that $0$ is a regular value of $\tilde f$ (see Section 2.3 of \cite{GP}).

The 1-dimensional case of Theorem \ref{brouwer} is more elementary than the general case - it follows from the intermediate value theorem.  Brouwer had the important insight that the same statement holds for any dimension $N$.  He used this insight to prove some important results in topology, including the Brouwer fixed point theorem and the topological invariance of dimension.

Here is a simple corollary of Theorem \ref{brouwer}, which is related to Theorem \ref{buvar}.  Suppose that $X$ is the closed unit ball $\bar B^N(1)$, and suppose that $g: X \rightarrow \RR^N$ is the identity.  Suppose that $f$ is a smooth map that agrees with $g$ on $\partial B^N(1)$.  Then it follows from the Theorem that $f$ vanishes at some point in $B^N$.  Indeed, suppose that $f^{-1}(0)$ was empty.  Then 0 would be a regular value of both $f$ and $g$.  But $\# f^{-1}(0) = 0$, and $\# g^{-1}(0) = 1$.  This contradiction shows that $f$ must vanish somewhere in the unit ball.  The proof of Theorem \ref{buvar} is based on a similar argument, but instead of using the boundary condition $f|_{\partial X} =g$, we instead use Condition \ref{f_vantipode}.

Now we begin the proof of Theorem \ref{buvar}.  Let us recall the setup.  Recall that $X$ is the product of spheres $X = \prod_{j=1}^s S^{2^{j-1}}$.  We note that $\Dim X = N = 2^s -1$.  We write a point $x \in X$ as $(x_1, ..., x_s)$ with $x_j \in S^{2^{j-1}}$.  We define the coordinate-flipping operation $Fl_j: X \rightarrow X$ by changing the sign of the $j^{th}$ coordinate:

$$ Fl_j( x_1, ..., x_{j-1}, x_j, x_{j+1}, ..., x_s) = ( x_1, ..., x_{j-1}, - x_j, x_{j+1}, ..., x_s). $$

For each $v \in \ZZ_2^s \setminus \{ 0 \}$, suppose that $f_v: X \rightarrow \RR$ is a continuous function that obeys the following antipodal-type condition \ref{f_vantipode}:

\begin{equation*}
 f_v (Fl_j x) = (-1)^{v_j} f_v (x) \textrm{ for all } j = 1, ..., s. \eqno{(\ref{f_vantipode})}
\end{equation*}

\noindent Note that we have $N=2^s - 1$ functions $f_v: X \rightarrow \RR$.  We let $f: X \rightarrow \RR^N$ be the function with coordinates $f_v$.  We want to conclude that $f^{-1}(0)$ is non-empty.

We will construct below a smooth function $g: X \rightarrow \RR^N$ obeying Condition \ref{f_vantipode}, so that 0 is a regular value of $g$, and so that $\# g^{-1}(0) = 2^s$.  The number $2^s$ here has to do with the symmetries coming from Condition \ref{f_vantipode}: if $g$ obeys Condition \ref{f_vantipode} and $g(x_1, ..., x_s) = 0$, then it follows that $g( \pm x_1 , \pm x_2, ..., \pm x_s) = 0$.  

Suppose that $H \subset X$ is a product of (closed) hemispheres.  If $H_j$ is a hemisphere of $S^{2^{j-1}}$, then $H = \prod_{j=1}^s H_j$.  If $x = (x_1, ..., x_s) \in X$, and we consider the $2^s$ points $(\pm x_1, ..., \pm x_s)$, then as long as none of these points lie on $\partial H$, exactly one of them lies in $H$.  So if $f$ obeys Condition \ref{f_vantipode}, and if $f^{-1}(0) \cap \partial H$ is empty, then

$$ \# f^{-1}(0) = 2^s \# ( f^{-1}(0) \cap H ) . $$

\noindent In particular, for a generic $H \subset X$, $\# ( g^{-1}(0) \cap H ) = 1$.  

To see how Theorem \ref{brouwer} is relevant, suppose that $h: X \rightarrow \RR^N$ is a smooth function obeying Condition \ref{f_vantipode} and so that 0 is a regular value of $h$, and so that $h = g$ on $\partial H$.  By Theorem \ref{brouwer}, $\# ( h^{-1} (0) \cap H)$ is odd, and so $\# h^{-1}(0)$ is an odd multiple of $2^s$.  In particular, $h^{-1}(0)$ is not empty.  Now not every function $f$ obeying Condition \ref{f_vantipode} agrees with $g$ on $\partial H$, but using Theorem \ref{brouwer} repeatedly, we can prove Theorem \ref{buvar}.

So let us suppose that $f: X \rightarrow \RR^N$ is a continuous function obeying Condition \ref{f_vantipode}, and suppose that $f^{-1}(0)$ is empty.  We can approximate $f$ by a smooth function $f_1: X \rightarrow \RR^N$ which still obeys Condition \ref{f_vantipode} and $f_1^{-1}(0)$ is still empty.  Since $f_1^{-1}(0)$ is empty, 0 is a regular value of $f_1$.

We will find a sequence of $H_i$ and maps $f_i: X \rightarrow \RR^N$ so that 

\begin{enumerate}

\item Each $H_i$ is a product of hemispheres as described above.

\item $f_{i+1}$ agrees with $f_i$ on $\partial H_i$.

\item Each $f_i$ obeys Condition \ref{f_vantipode}.

\item 0 is a regular value of each $f_i$, and the function $f_i$ does not vanish on $\partial H_i$.

\item For some large $i$, $f_i = g$.

\end{enumerate}

Applying Theorem \ref{brouwer}, we see that $\# (f_i^{-1}(0) \cap H_i) = \# (f_{i+1}^{-1}(0) \cap H_i)$ modulo 2.  Therefore,

$$ \frac{ \# f_i^{-1}(0)} {2^s} =  \frac{ \# f_{i+1}^{-1}(0)} {2^s} \textrm{ modulo 2}. $$

\noindent Since $\# g^{-1}(0) = 2^s$, we see that $\# f_1^{-1}(0)$ is an odd multiple of $2^s$, and in particular $f_1^{-1}(0)$ is not empty.

To finish the proof, it remains to construct the maps $f_i$, and to construct the map $g$.  Constructing the maps $f_i$ is straightforward.  We know that $f_i^{-1}(0)$ is a finite set.  Pick a product of hemispheres $H_i$ so that $\partial H_i$ does not intersect $f_i^{-1}(0)$.  Now we define $f_{i+1}$ on $H_i$ as follows.  We define open sets $\partial H_i \subset U_1 \subset U_2 \subset H_i$, where $U_2$ is a small neighborhood of $\partial H_i$.  We let $f_{i+1}$ agree with $g$ on $H_i \setminus U_2$, and we let it agree with $f_i$ on $U_1$.  On the region $U_2 \setminus U_1$, we define $f_{i+1}$ in a smooth way so that 0 is a regular value of $f_{i+1}$.  (It is always possible to do this.  In fact, given any smooth extension $f_{i+1, 0}$, we can always slightly perturb $f_{i+1,0}$ in a small neighborhood of $U_2 \setminus U_1$ so that 0 will become a regular value -- cf. the Extension Theorem on page 72 of \cite{GP}.)  We have now defined $f_{i+1}$ on $H_i$.  Since $f_{i+1} = f_i$ on a neighborhood of $\partial H_i$, $f_{i+1}$ obeys Condition \ref{f_vantipode} on $\partial H_i$.  Therefore, we can extend $f_{i+1}$ to all of $X$ in a way that obeys Condition \ref{f_vantipode}.  We have now defined $f_{i+1}$ and we see that it has properties 1-4.  Also $f_{i+1}$ agrees with $g$ except on a small neighborhood of $H_1 \cap H_2 \cap ... \cap H_i$.  For a large value of $i$, we can arrange that $H_1 \cap H_2 \cap ... \cap H_i$ is empty, and so $f_{i+1} = g$.  

It just remains to construct the function $g: X \rightarrow \RR^{2^s -1}$.  We need $g$ to be a smooth function obeying the antipodal condition \ref{f_vantipode} and so that:

\begin{itemize}

\item $g$ vanishes at exactly $2^s$ points of $x$.

\item At each point $x$ where $g$ vanishes, $dg_x: T_x X_s \rightarrow \RR^{2^s-1}$ is surjective.

\end{itemize}

A point $x \in X_s$ has the form $x = (x_1, ..., x_s)$ with $x_j \in S^{2^{j-1}} \subset \RR^{2^{j-1} + 1}$.  For each $j$, we will choose coordinates on $\RR^{2^{j-1}+1}$.  We will write $g$ in those coordinates.  Our function $g$ will have the form:

$$ g_v(x) = \prod_{j: v_j = 1} (\textrm {one of the coordinates of $x_j$}) . $$

This form guarantees that $g_v( Fl_j x) = (-1)^{v_j} g_v(x)$.

Here is a choice of coordinates that allows us to make a clean analysis of the situation.  
For each $j = 1, ..., s$, let $t_j$ be one of the coordinates on $\RR^{2^{j-1} + 1}$.  We still have to give names to $2^{j-1}$ other coordinates on $\RR^{2^{j-1}+1}$.  
For each $v \in \ZZ_2^s \setminus \{ 0 \}$, let $j(v) := \max \{ j | v_j = 1\}$.  
For each $v \in\ZZ_2^s \setminus \{ 0 \}$, let $x_v$ be a coordinate on $\RR^{2^{j-1}+1}$.  
There are $2^{j-1}$ different $v \in \ZZ_2^s \setminus \{ 0 \}$ with $j(v) = j$, so we get $2^{j-1}$ coordinates on $\RR^{2^{j-1}+1}$.  

For example, if $s = 3$, then the coordinates are as follows.

When $j=1$, the coordinates on $\RR^{2^{1-1} + 1} = \RR^2$ are $t_1$ and $x_{(1,0,0)}$.  

When $j=2$, the coordinates on $\RR^{2^{2-1} +1} = \RR^3$ are $t_2$ and $x_{(0,1,0)}$ and $x_{(1,1,0)}$.

When $j=3$, the coordinates on $\RR^{2^{3-1} +1} = \RR^5$ are $t_2$ and $x_{(0,0,1)}$,$x_{(1,0,1)}$, $x_{(0,1,1)}$, and $x_{(1,1,1)}$.

With these coordinates, we can define $g_v$:

\begin{equation} \label{f_vformula} g_v(x) = x_v \left( \prod_{j: v_j = 1 \textrm{ and } j < j(v)} t_j \right).
\end{equation}

We claim that if $g_v (x) = 0$ for all $v \in \ZZ_2^s \setminus \{ 0 \}$, then $x_v = 0$ for all $v$, and $t_j \not=0$ for all $j$.  We prove this by induction on $j$, starting with $j=1$.  For $j=1$, we have two coordinates on $\RR^{2^{j-1}+1} = \RR^2$.  These are $t_1$, and $x_{e_1}$ where $e_1 = (1, 0, ..., 0) \in \ZZ_2^s$.  Now $g_{e_1}(x) = x_{e_1} = 0$.  Since $x_{e_1} = 0$ and $(t_1, x_{e_1}) \in S^1$, we have $t_1 \not=0$.  This proves the case $j=1$, giving the base case for the induction.  Suppose that $x_v = 0$ for all $v$ with $j(v) < j_0$, and $t_j \not= 0$ for all $j < j_0$.  Next we will prove that $x_v = 0$ for all $v$ with $j(v) = j_0$.  Suppose that $j(v) = j_0$.  By equation \ref{f_vformula}, we see that $g_v(x)$ is equal to $x_v$ times a product of some $t_j$'s with $j < j_0$.  Since these $t_j$'s are all non-zero, we have $x_v = 0$.  But now $(t_{j_0}, 0, ...., 0) \in S^{2^{j-1}} \subset \RR^{2^{j-1} + 1}$, and so $t_{j_0} \not= 0$.  

So the set $\{ x \in X | g_v(x) = 0 \textrm{ for all } v \}$ is the set of points with coordinates $t_j = \pm 1$ for all $j$ and $x_v = 0$ for all $v$.  The number of points in this set is $2^s$.

Now we have to check that 0 is a regular value for $g$.  Let $p$ be a point of $g^{-1}(0)$.  At the point $p$, $t_j = \pm 1$ for all $j$ and $x_v = 0$ for all $v$.   The tangent space $T_p X_s$ is the plane $t_j = 0$ for all $j$.  This plane has coordinates $x_v$.  In these coordinates, the derivative of $g$ has a very simple form.  If $v \not= v'$, then $\frac{ \partial g_v }{\partial x_{v'}} = 0$.  If $v = v'$, then $\frac{ \partial g_v }{\partial x_{v'}} = \prod_{j: v_j = 1 \textrm{ and } j < j(v)} t_j  = \pm 1$.  In short, the matrix $dg_p$ is a diagonal matrix with diagonal entries $\pm 1$.  Therefore, $dg_p$ is surjective, and so 0 is a regular value of $g$.

\end{document}